\documentclass[a4paper]{article}

\usepackage[utf8]{inputenc}

\usepackage{lmodern}
\usepackage{array}
\usepackage{graphicx}
\usepackage{subfig}
\usepackage{amssymb, amsfonts, amsthm, amsmath}
\usepackage{enumerate}

\usepackage[english]{babel}

\usepackage[colorlinks]{hyperref}

\usepackage{tikz}
\numberwithin{equation}{section}

\theoremstyle{plain}
\newtheorem{theorem}{Theorem}[section]
\newtheorem{corollary}[theorem]{Corollary}

\newtheorem{lemma}[theorem]{Lemma}
\newtheorem{proposition}[theorem]{Proposition}

\theoremstyle{definition}

\theoremstyle{remark}

\newcommand{\N}{\mathbb{N}}

\newcommand{\R}{\mathbb{R}}

\newcommand{\ind}[1]{\mathbf{1}_{\left\{#1\right\}}}

\DeclareMathOperator{\E}{\mathbf{E}}
\renewcommand{\P}{\mathbf{P}}

\renewcommand{\bar}[1]{\overline{#1}}

\renewcommand{\epsilon}{\varepsilon}
\renewcommand{\phi}{\varphi}

\newcommand{\dd}{\mathrm{d}}

\title{On the length of the shortest path in a sparse Barak-Erd\H{o}s graph}
\author{Bastien Mallein\footnote{LAGA UMR 7539, Université Sorbonne Paris Nord, and DMA, UMR 8553, École Normale Supérieure, Email: \texttt{mallein@math.univ-paris13.fr}} \and Pavel Tesemnikov\footnote{Novosibirsk State University, Sobolev Institute of Mathematics and MCA, Email: \texttt{tesemnikov.p@gmail.com}}}
\date{\today}

\begin{document}

\maketitle

\begin{abstract}
We consider an inhomogeneous version of the Barak-Erd\H{o}s graph, i.e. a directed Erd\H{o}s-Rényi random graph on $\{1,\ldots,n\}$ with no loop. Given $f$ a Riemann-integrable non-negative function on $[0,1]^2$ and $\gamma > 0$, we define $G(n,f,\gamma)$ as the random graph with vertex set $\{1,\ldots,n\}$ such that for each $i < j$ the directed edge $(i,j)$ is present with probability $ p_{i, j}^{(n)} = \frac{f(i/n,j/n)}{n^\gamma}$, independently of any other edge. We denote by $L_n$ the length of the shortest path between vertices $1$ and $n$, and take interest in the asymptotic behaviour of $L_n$ as $n \to \infty$.
\end{abstract}

\section{Introduction}

The Barak-Erd\H{o}s graph is a random directed graph with no loop constructed in the following fashion. Given $n \in \N$ and $p \in (0,1)$, the Barak-Erd\H{o}s graph $G(n,p)$ is a graph with vertex set $\{1,\ldots,n\}$ such that for each $i < j$, the edge $(i,j)$ from vertex $i$ to vertex $j$ is present with probability $p$, independently of any other directed edge. This graph is a directed acyclic version of the well-known Erd\H{o}s-R\'{e}nyi graph. It can be used to model community food webs in ecology \cite{NeC86}, or the task graph for parallel processing in computer sciences \cite{GNPT86}.

In particular, the length (number of edges) of the longest (directed) path, denoted $M_n$, has been the subject of multiple studies, as $M_n+1$ is the number of steps needed to complete the task graph assuming maximal parallelization. Newman \cite{New92} proved that $\frac{M_n}{n}$ converges in law to a deterministic function $p \mapsto C(p)$. Increasingly precise bounds were obtained on this function and its generalizations by \cite{FoK03,ChR17,FoK18,MaR19,MaR21,FKMR}.

In the present article, we take interest in the length $L_n$ of the \emph{shortest} path between vertices $1$ and $n$ in this graph, which has been much less studied. It is worth noting that for fixed value of $p$, one has
\[
  \P(L_n = 1) = p \quad \text{and} \quad \lim_{n \to \infty} \P(L_n = 2) = 1-p,
\]
as with probability $1 - (1 - p^2)^{n-2}$, there is a vertex $j \in \{2,\ldots n-1\}$ connect to both $1$ and $n$, hence $L_n$ is equal to $1$ or $2$ with high probability. In particular, the length of the shortest path in dense graphs remains tight.

This fact is mentioned in \cite{Tes18}, which takes interest in the asymptotic behaviour of $L_n^{(\gamma)}$, the length of the shortest path between vertices $1$ and $n$ in a graph with connexion constant $p_n = n^{-\gamma}$, i.e. in the limit of sparse graphs. There, it is shown that for all $k \geq 2$,
\[
\lim_{n \to \infty} \P (L^{(\gamma)}_{n} \le k) = 0 \quad \text{ if } 1 - \frac{1}{k} < \gamma.
\]
We extend this result in the present article by obtaining the convergence in distribution of $L_n^{(\gamma)}$ for all $\gamma \in (0,1)$.

We consider here the asymptotic behaviour of the length of the shortest path between $1$ and $n$ in a time-inhomogeneous version of the Barak-Erd\H{o}s graph, defined as follows. Let $f$ be a Riemann-integrable positive function on $[0,1]^2$ and $\gamma \in (0,1)$. For each $n \in \N$ and $i < j$, we set $p_{i,j}^{(n)} = \frac{f(i/n,j/n)}{n^\gamma}$. The time-inhomogeneous sparse Barak-Erd\H{o}s graph $G(n,f,\gamma)$ is defined as a graph with vertex set $\{1,\ldots,n\}$ such that for each $i < j$, the directed edge $(i,j)$ is present with probability $p_{i,j}^{(n)}$.

The main result of the article is the following estimate on the asymptotic behaviour of $L_n$ for $\gamma = 1-\frac{1}{k}$.
\begin{theorem}
\label{thm:main}
Let $f$ be a Riemann-integrable non-negative function on $[0,1]^2$ and $k \in \N$. We fix $\gamma = 1 - \frac{1}{k}$ and we set
\[
  c_k(f) = \int_{0 < u_1 < \ldots<u_{k-1}<1} \prod_{j=0}^{k-1} f(u_j,u_{j+1}) \dd u_1 \cdots \dd u_{k-1} \in [0,\infty],
\]
with $u_0=0$ and $u_1=1$. Writing $L_{n}$ for the length of the shortest path in $G(n,f,\gamma)$, we have
\[
  \lim_{n \to \infty} \P(L_n = k+1) = 1 - \P(L_n=k) = \exp\left( - c_k(f) \right).
\]
\end{theorem}

By coupling arguments, Theorem~\ref{thm:main} can be extended to describe the convergence in distribution of $L_n$ as $n \to \infty$ for any time-inhomogeneous Barak-Erd\H{o}s graph.
\begin{corollary}
\label{cor:coupling}
Let $f$ be a Riemann-integrable positive function on $[0,1]^2$ and $\gamma > 0$. For $k \geq 2$, we have
\[
  \lim_{n \to \infty} \P(L_n \leq k) = \begin{cases}
      0 & \text{ if } k < \frac{1}{1-\gamma},\\
      e^{-c_k(f)} & \text{ if } k = \frac{1}{1-\gamma},\\
      1 & \text{ otherwise.}
    \end{cases}
\]
%Let $L_n$ be the length of the shortest path in $G(n,f,\gamma)$, for each $k \in \N$ such that $\gamma > 1 - \frac{1}{k}$ or $c_k(f) = 0$, we have
%\[
%  \lim_{n \to \infty} \P(L_n \leq k) = \begin{cases}
%    0 & \text{ if } k < \frac{1}{1-\gamma} \text{ or } c_k(f) = 0\\
%    1 & \text{ otherwise.}
%  \end{cases}
%\]
\end{corollary}

Observe that for $\gamma = 1$, the Barak-Erd\H{o}s graph becomes unconnected, so that $L_n = \infty$ with positive probability. In the present article, we do not treat the case $n p_n \to \infty$ with $n^{1 - \epsilon}p_n \to 0$ for all $\epsilon > 0$. However, a phase transition should be observed for the asymptotic behaviour of $L_n$ when $p_n \approx \frac{\log n}{n}$, as the graph becomes disconnected.

\subsection{Some examples and applications}

A class of inhomogeneous Barak-Erd\H{o}s graphs previously studied strongly inhomogeneous graph. In this class of graphs, the probability of presence of the edge $(i,j)$ is given by $\theta \frac{(j-i)^\alpha}{n^\beta}$, with $\theta >0$,  $\beta > 0$ and $\alpha \in (-1,\beta)$. This model can be constructed as an inhomogeneous Barak-Erd\H{o}s graph, setting $f(x,y) = \theta (y-x)^\alpha$ and $\gamma = \beta - \alpha$. Applying Corollary~\ref{cor:coupling}, we observe that for any $k \geq 2$, if $1-\frac{1}{k-1} < \beta - \alpha < 1 - \frac{1}{k}$, we have $L_n \to k$ in probability. Additionally, if $\beta - \alpha = 1 - \frac{1}{k}$, we set
\[
  c = \theta^k \int_{S_k} \prod_{j=1}^k t_j^\alpha \dd t_1 \cdots \dd t_{k-1} = \frac{\theta^k \Gamma(1+\alpha)^k}{\Gamma(k(1+\alpha))},
\]
where $S_k = \{(t_1,\ldots t_k) : t_1+\cdots +t_k = 1\}$. We conclude by Theorem~\ref{thm:main} that $L_n$ converges in distribution to $e^{-c}\delta_k + (1 - e^{-c})\delta_{k+1}$ as $n \to \infty$.

Remark that using the coupling given in Proposition~\ref{prop:coupling}, for a similar model with $\alpha \leq -1$, we can obtain
\[
  \lim_{n \to \infty} L_n = k \text{ in probability if } \beta - \alpha = \left[1 - \frac{1}{k-1}, 1 - \frac{1}{k}\right).
\]

This result is an extension of Tesemnikov's \cite{Tes18} estimates on the length of the shortest path in the inhomogeneous Barak-Erd\H{o}s graph, setting $\beta = 0$. Outside of the boundary cases, the convergence in probability of $L_n$ to $k \in \N$ can be obtained through first- and second-moment methods, using see Lemma~\ref{lem:mean}. We handle the boundary cases by using the Chen-Stein method, showing in Lemma \ref{S-C method} that the law of the number of paths of length $k$ is close to a Poisson distribution for $n$ large enough.

An other example of interest in the case of Barak-Erd\H{o}s graphs with exponential density of connexion. Setting $\lambda,\mu \in \R$ and $\gamma > 0$, we consider a Barak-Erd\H{o}s graph with $p_{i,j}^{(n)} = \frac{e^{\lambda(j-i)/n + \mu}}{n^\gamma}$. In this case, we have
\[
  c_k(f) = \int_{0 < u_1 < \ldots<u_{k-1}<1} \prod_{j=0}^{k-1} e^{\lambda(u_{j+1}-u_j) + \mu} \dd u_1 \cdots \dd u_{k-1} = \frac{e^{\lambda + k \mu}}{k!},
\]
and Theorem~\ref{thm:main} and Corollary~\ref{cor:coupling} apply.

\section{Proof of the main result}

For each $k \leq n$, we denote by $\Gamma_k(n) = \{ \rho \in \N^{n+1} : \rho_0 = 1 < \rho_1 < \cdots < \rho_{k-1} < \rho_k = n \}$ the set of increasing paths from $1$ to $n$. As a first step towards estimating the length of the shortest in a time-inhomogeneous Barak-Erd\H{o}s graph, we compute the mean number of paths of length $k$.

\begin{lemma}
\label{lem:mean}
Let $\gamma > 0$ and $f$ a Riemann-integrable non-negative function. For $n \in \N$, we write $Z_n(k)$ for the number of paths of length $k$ between $1$ and $n$ in $G(n,f,\gamma)$, we have
\[
  \E(Z_n(k)) \sim n^{(k-1) - k \gamma} c_k(f).
\]
\end{lemma}

\begin{proof}
By linearity of the expectation, we have
\begin{align*}
  \E(Z_n(k))
  &= \sum_{\rho \in \Gamma_k(n)} \prod_{j = 0}^{k-1} p^{(n)}_{\rho_{j},\rho_{j+1}}\\
&= n^{(k-1) - k \gamma} \frac{1}{n^{k-1}}  \sum_{\rho \in \Gamma_k(n)}\prod_{j = 0}^{k-1} f\left( \tfrac{\rho_j}{n}, \tfrac{\rho_{j+1}}{n} \right).
\end{align*}
Then $\lim_{n \to \infty} \frac{1}{n^{k-1}} \sum_{\rho \in \Gamma_k(n)}\prod_{j = 0}^{k-1} f\left( \tfrac{\rho_{j}}{n}, \tfrac{\rho_{j+1}}{n} \right) = c_k(f)$, by Riemann integration, which completes the proof.
\end{proof}

In particular, we remark that under the assumptions of Theorem~\ref{thm:main}, the mean number of paths of length $k$ in $G(n,f,\gamma)$ converges to $c_k(f)$. Using this observation, we now prove that the number of paths of length $k$ converges to a Poisson variable.
\begin{lemma} \label{S-C method}
With the notation and assumptions of Theorem~\ref{thm:main}, we have
\[
  \lim_{n \to \infty} Z_n(k) = \mathcal{P}\left( c_k(f) \right) \quad \text{ in distribution.}
\]
\end{lemma}

\begin{proof}
We use the Chen-Stein method \cite{Che75,Ste72} to prove the convergence in distribution of $Z_n(k)$. More precisely, we show that for all $j \in \N$ we have
\begin{equation}
  \label{eqn:aim}
  \lim_{n \to \infty} j \P(Z_n(k)  = j) - \E(Z_n(k)) \P(Z_n(k) = j-1) =0.
\end{equation}
Together with a tightness argument (due to the fact that $\E(Z_n(k))$ converges), it proves that $Z_n(k)$ converges in distribution to a Poisson variable with parameter $\lim_{n \to \infty} \E(Z_n(k)) =c_k(f)$.

Let $j \in \N$, we rewrite
\begin{equation}
  \label{eqn:formula}
  j \P(Z_n(k)=j)
  = \E\left( \sum_{\rho \in \Gamma_k(n)} \ind{\rho \text{ open}} \ind{Z_n(k)=j}  \right),
\end{equation}
where $\rho$ is said to be open if all edges $(\rho_{i},\rho_{i+1})$ are present in the graph. Moreover for all $\rho \in \Gamma_k(n)$, we have
\begin{multline*}
  \left| \P(Z_n(k)=j|\rho \text{ open}) - \P(Z_n(k) = j-1) \right|\\
   \leq \P\left( \text{exists a path of length $k$ between $1$ and $n$ sharing an edge with $\rho$} \right).
\end{multline*}
Indeed, to construct a graph with same law as $G(n,f,\gamma)$ conditionally on $\rho$ being open, it is enough to add to the graph $G(n,f,\gamma)$ the edges $(\rho_j,\ldots \rho_{j+1})$ for all $1 \leq j \leq n$ if these are not already present. If opening these edges creates new paths, then these path would have to share at least one edge with $\rho$.

We remark that if there exists a path of length $k$ between $1$ and $n$, there exists $0 \leq i < j \leq k$ and $2 \leq \ell < k$ such that there exists a path of length $\ell$ between $\rho_i$ and $\rho_j$ that does not intersect $\rho$. Writing $Y_{i,j,\ell}$ the number of such paths, with the same method as in Lemma~\ref{lem:mean}, we compute
\begin{align*}
  \E\left( Y_{i,j,\ell}\right) &= \sum_{\rho_i < \bar{\rho}_1 < \cdots < \bar{\rho}_{\ell-1} < \rho_j} \prod_{q= 0}^{\ell-1} p^{(n)}_{\bar{\rho}_{q},\bar{\rho}_{q+1}}\\
  &\leq  n^{-\gamma \ell} \sum_{\bar{\rho} \in \Gamma_k(\ell)}\prod_{j = 0}^{\ell-1} f\left( \tfrac{\bar{\rho}_q}{n}, \tfrac{\bar{\rho}_{q+1}}{n} \right) \to 0 \text{ as $n \to \infty$.}
\end{align*}
Therefore, by union bound, we deduce that
\[
  \lim_{n \to \infty} \P(Z_n(k) = j | \rho \text{ open}) - \P(Z_n(k)= j-1) = 0,
\]
which then yields by \eqref{eqn:formula}
\[
  \sum_{\rho \in \Gamma_k(n)} \P(Z_n(k) = j \text{ and }\rho\text{ open}) - \P(Z_n(k)=j-1) \E(Z_n(k)) = o(\E(Z_n(k))).
\]
As $\E(Z_n(k))$ is bounded, we obtain \eqref{eqn:aim}.

We remark that $\sup_{n \in \N}\E(Z_n(k)) < \infty$, hence $(Z_n(k))$ is tight. Consider any subsequence $(n_j)$ so that $Z_{n_j}(k)$ converges in distribution as $j \to \infty$. Writing $Y$ a random variable with this distribution, we have for all $j \in \N$:
\[
  j \P(Y=j) = c_k(f) \P(Y=j-1),
\]
using that $\E(Z_{n_j}(k)) \to c_k(f)$. Hence $\P(Y=j) = \frac{c_k(f)^j}{j!} \P(Y=0)$, with $\P(Y>n) \to 0$ as $n \to \infty$. We conclude that $Y$ is a $\mathcal{P}(c_k(f))$ random variable.

As any converging subsequence of $(Z_n(k))$ is converging to $\mathcal{P}(c_k(f))$ in law, we conclude that $Z_n(k)$ converges to $\mathcal{P}(c_k(f))$ in law as $n \to \infty$.
\end{proof}

Before turning to the corollary, we introduce the following coupling estimate, which loosely states that a more connected graph will have a shorter shortest path between $1$ and $n$.
\begin{proposition}
\label{prop:coupling}
Let $G_n$, $\bar{G}_n$ be two inhomogeneous Barak-Erd\H{o}s graphs such that an edge between $i$ and $j$ is present with probability $p^{(n)}_{i,j}$ and $\bar{p}^{(n)}_{i,j}$ respectively. If $p^{(n)}_{i,j} \leq \bar{p}^{(n)}_{i,j}$ for any $ i $ and $ j $, then there exists a coupling between $G_n$ and $\bar{G}_n$ such that $L_n \geq \bar{L}_n$.
\end{proposition}

\begin{proof}
We assume $\bar{G}_{n}$ to be constructed on some probability space. Take any existing edge $(i, j)$ of $\bar{G}_{n}$ and do the following procedure: chosen edge is stayed in graph with probability $p_{i,j}^{(n)}/\bar{p}_{i,j}^{(n)}$ and deleted with remained probability. This procedure creates a random graph distributed exactly as $G_{n}$ and is a subgraph of $\bar{G}_n$. Therefore, as no new edge was added, the length of the shortest path cannot have decrease.
\end{proof}

\begin{proof}[Proof of Corollary~\ref{cor:coupling}]
We assume first that $k < \frac{1}{1-\gamma}$. Then, by Lemma~\ref{lem:mean}, we have
\[
  \lim_{n \to \infty} \sum_{j=1}^k \E(Z_n(j)) =0,
\]
therefore  $\P(L_n \leq k) \to 0$ by Markov inequality.

The case $k = \frac{1}{1-\gamma}$ is covered by Theorem~\ref{thm:main}.

Finally, if $k > \frac{1}{1-\gamma}$, then for all $A > 0$, the Barak-Erd\H{o}s graph $G(n,f,\gamma)$ can be coupled with $G(n,Af,\frac{k-1}{k})$ for $n$ large enough, using Proposition \ref{prop:coupling}. Therefore
\[
  \liminf_{n \to \infty}\P(L_n \leq k) \geq 1 - e^{-A^k c_k(f)},
\]
using Theorem~\ref{thm:main} and that $c_k(Af) = A^k c_k(f)$. As $f$ is positive, $c_k(f)$ is positive, and letting $A \to \infty$ we conclude that $\P(L_n \leq k) \to 1$.
\end{proof}

\paragraph{Acknowledgements.}
The research of both authors was supported by joint grant (19-51-15001) of the Russian Foundation for Basic Research and from the French CNRS' PRC grant.

\bibliographystyle{plain}
\bibliography{biblio.bib}
\end{document}